\documentclass[a4paper,12pt]{article}

\usepackage{mymacros}

\newcommand{\Hmf}{\operatorname{Hmf}}
\newcommand{\CM}{\operatorname{CM}}
\newcommand{\Mat}{\operatorname{Mat}}

\newcommand{\Yvo}{\Yv^{\circ}}
\newcommand{\dirscE}{\scE^{\to}}
\newcommand{\dirscF}{\scF^{\to}}

\title{Hyperplane sections and stable derived categories}
\author{Kazushi Ueda}
\date{}
\begin{document}

\maketitle

\begin{abstract}
We discuss the relation
between the graded stable derived category of a hypersurface and
that of its hyperplane section.
The motivation comes from the compatibility
between homological mirror symmetry
for the Calabi-Yau manifold
defined by an invertible polynomial
and that
for the singularity
defined by the same polynomial.
\end{abstract}

\section{Introduction}

Let $(Y, \scO_Y(1))$ be a polarized smooth projective variety
of dimension $d$
over a field $\bsk$,
and $X = s^{-1}(0)$ be a smooth hypersurface
defined by a section $s \in H^0(\scO_Y(a))$
of degree $a$.
For coherent sheaves $\scE$ and $\scE'$ on $X$
restricted from those on $Y$,
the push-forward functor
$
 \iota_* : D^b \coh X \to D^b \coh Y
$
along the inclusion $\iota : X \hookrightarrow Y$
satisfies
\begin{align} \label{eq:ext_coh}
 \Hom_Y^i(\iota_* \scE, \iota_* \scE')
  &\cong \Hom_X^i(\scE, \scE')
   \oplus \Hom_X^{d-i}(\scE', \scE \otimes \omega_Y|_X)^\vee,
\end{align}
where
$\omega_Y|_X = \iota^* \omega_Y$ is the restriction
of the dualizing sheaf of $Y$ to $X$
and $\bullet^\vee$ denotes the $\bsk$-dual vector space.
In particular,
if $Y$ has the trivial dualizing sheaf,
then one has
$$
 \Hom_Y^i(\iota_* \scE, \iota_* \scE')
  \cong \Hom_X^i(\scE, \scE')
   \oplus \Hom_X^{d-i}(\scE', \scE)^\vee.
$$
If $\omega_Y \cong \scO_Y(r)$ for some $r \in \bZ$ and
$H^i(\scO_Y(n)) = 0$ for any $n \in \bZ$
unless $i = 0, d$,
then the graded ring
$
 \Sbar = \bigoplus_{n=0}^\infty H^0(\scO_Y(n))
$
is {\em Gorenstein} with {\em $a$-invariant} $r$,
in the sense that $S$ has finite injective dimension as a graded $S$-module
and the graded canonical module $K_S$
is isomorphic to the free module $S(r)$
\cite[Lemma 2.12]{Orlov_DCCSTCS}.
The latter condition is equivalent to the isomorphism
$
 \RHom_\Sbar(\bsk, \Sbar) \cong \bsk(-r)[-d-1].
$
Here the round bracket and the square bracket indicate
the shift in the {\em internal} and {\em homological} grading
respectively.

The {\em graded stable derived category}
of $\Sbar$ is defined as the quotient category
$$
 D^b_\sing(\gr \Sbar)
  := D^b(\gr \Sbar) / D^\perf(\gr \Sbar)
$$
of the bounded derived category
of finitely-generated graded $\Sbar$-modules
by the full triangulated subcategory $D^\perf(\gr \Sbar)$
consisting of perfect complexes
(i.e., bounded complexes of finitely-generated projective modules)
\cite{Buchweitz_MCM,Happel_GA,Krause_SDCNS,Orlov_TCS}.
Orlov \cite[Theorem 2.13]{Orlov_DCCSTCS} has shown the existence of
\begin{itemize}
 \item
a full and faithful functor
$\Psi_\Sbar : D^b_\sing(\gr \Sbar) \hookrightarrow D^b \coh Y$
if $r < 0$,
 \item
an equivalence
$
 \Psi_\Sbar :D^b_\sing(\gr \Sbar) \simto D^b \coh Y
$
if $r = 0$, and
 \item
a full and faithful functor
$
 \Psi_\Sbar : D^b \coh Y \hookrightarrow D^b_\sing(\gr \Sbar)
$
if $r > 0$.
\end{itemize}

The graded ring
$
 \Rbar = \bigoplus_{n=0}^\infty H^0(\scO_X(n))
$
is the quotient ring of $\Sbar$
by the principal ideal generated by $s$.
Let $\Phi_{\gr} : \gr \Rbar \to \gr \Sbar$ be the functor
sending a graded $\Rbar$-module
to the same module
considered as a graded $\Sbar$-module.
Since $\Rbar$ is perfect as an $\Sbar$-module,
the functor $\Phi_{\gr}$ induces the push-forward functor
$$
 \Phi_\sing : D^b_\sing(\gr \Rbar) \to D^b_\sing(\gr \Sbar)
$$
in the stable derived categories.
When the $a$-invariant of $\Rbar$ is $1$ and
$Y$ is Calabi-Yau,
then
\cite[Theorem 1.1]{Kobayashi-Mase-Ueda_EUS}
states that the functor
$$
 \Psi_\Sbar \circ \Phi_\sing \circ \Psi_\Rbar
  : D^b \coh X \to D^b \coh Y
$$
is isomorphic to the push-forward functor
$$
 \iota_* : D^b \coh X \to D^b \coh Y.
$$
In particular,
if two objects $\scE, \scE'$ in $D^b_\sing(\gr \Rbar)$
are in the image of $\Psi_\Rbar$,
then \eqref{eq:ext_coh} implies
\begin{align} \label{eq:ext_K3}
 \Hom^i(\Phi_\sing(\scE), \Phi_\sing(\scE'))
  \cong \Hom^i(\scE, \scE')
   \oplus \Hom^{d-i}(\scE', \scE)^\vee.
\end{align}
On the other hand,
the semiorthogonal complement
of the image $\Psi_\Rbar$
is generated by the structure sheaf $\Rbar / \frakm_{\Rbar}$ of the origin,
which goes to $\scO_Y[d]$
by the composition $\Psi_\Sbar \circ \Phi_\sing$.
Since $\Rbar / \frakm_{\Rbar}$ is exceptional and $\scO_Y[d]$ is spherical,
\eqref{eq:ext_K3} holds also in this case.

In this paper,
we prove
a graded stable derived category analog of \eqref{eq:ext_coh}
in the following situation:
Let $R$ be a graded regular ring with $a$-invariant $r$
and Krull dimension $d$, and
$S = R \otimes_\bsk \bsk[w]$ be the tensor product of $R$
with the polynomial ring in one variable $w$
of degree $a$.
We identify the ring $R$
with its image by the natural injection $R \hookrightarrow S$.
Let further $f \in R_h$ be a homogeneous element of $R$ of degree $h$,
and $F = f + w g \in S_h$ be a homogeneous element of $S$
of the same degree as $f$.
We will always assume that $g \in w S$.
The corresponding quotient rings will be denoted by
$\Sbar = S / (F)$ and $\Rbar = R / (f) \cong \Sbar / (w)$.

\begin{theorem} \label{th:main}
For any objects $\scE$ and $\scE'$ of $D^b_\sing(\gr \Rbar)$,
one has
$$
 \Hom^i(\Phi_\sing(\scE), \Phi_\sing(\scE'))
  \cong \Hom^i(\scE, \scE')
   \oplus \Hom^{d-i}(\scE', \scE(r+h-a))^\vee.
$$
\end{theorem}


In particular,
when the $a$-invariant $r+h$ of $\Rbar$
and the degree $a$ of the variable $w$ coincides
(i.e., when $\Sbar$ has the trivial canonical module),
then one has
\begin{align} \label{eq:CY_completion}
 \Hom^i(\Phi_\sing(\scE), \Phi_\sing(\scE'))
  \cong \Hom^i(\scE, \scE')
   \oplus \Hom^{d-i}(\scE', \scE)^\vee.
\end{align}

The motivation for Theorem \ref{th:main} comes
from the compatibility
between homological mirror symmetry
for the singularity
defined by an invertible polynomial
and that for the Calabi-Yau manifold
defined by the same polynomial.
An integer $(d+1) \times (d+1)$-matrix
$(a_{ij})_{i, j = 1}^{d+1}$ with non-zero determinant
defines a polynomial $f \in \bsk[x_1, \dots, x_{d+1}]$ by
$$
 f = \sum_{i = 1}^{d+1} x_1^{a_{i 1}} \cdots x_{d+1}^{a_{i, d+1}},
$$
which is called {\em invertible}
if it has an isolated critical point at the origin.
Invertible polynomials play essential role
in transposition mirror symmetry
of Berglund and H\"{u}bsch
\cite{Berglund-Hubsch},
which has attracted much attention recently
(see e.g. \cite{Borisov_BHMSVA, Chiodo-Ruan_LGCYSSI, Krawitz}
and references therein).
The quotient ring
$\Rbar = \bsk[x_1, \dots, x_{d+1}] / (f)$ is naturally graded
by the abelian group $L$
generated by $d + 2$ elements $\vecx_i$ and $\vecc$ with relations
$$
 a_{i 1} \vecx_1 + \dots + a_{i, {d+1}} \vecx_{d+1} = \vecc,
  \qquad i = 1, \dots, d+1.
$$
Homological mirror symmetry
\cite{Kontsevich_HAMS}
for invertible polynomials
\cite{Takahashi_WPL}
is a conjectural equivalence
\begin{align} \label{eq:hms_sing}
 D^b_\sing(\gr \Rbar) \cong D^b \Fuk \fv
\end{align}
of triangulated categories.
Here $\Fuk \fv$ is the {\em Fukaya-Seidel category}
\cite{Seidel_PL}
of the exact symplectic Lefschetz fibration
obtained by Morsifying the Berglund-H\"{u}bsch transpose
$
 \fv = \sum_{i = 1}^{d+1} x_1^{a_{1i}} \cdots x_{d+1}^{a_{d+1,i}}
$
of $f$.
The equivalence \eqref{eq:hms_sing} is proved
when $\fv$ is the Sebastiani-Thom sum
of polynomials of types A or D
\cite{Futaki-Ueda_BP, Futaki-Ueda_Dn}.

Assume that one can add one more term
to $f$ and obtain another invertible polynomial
$$
 F = f + x_1^{a_{d+2, 1}} \cdots x_{d+2}^{a_{d+2, d+2}}
  \in \bsk[x_1, \ldots, x_{d+2}],
$$
with a suitable $L$-grading on $x_{d+2}$
such that $F$ is homogeneous of degree $\vecc$
and the quotient ring
$
 \Sbar = \bsk[x_1, \ldots, x_{d+2}] / (F)
$
is Gorenstein with the trivial $a$-invariant;
$
 K_\Sbar \cong \Sbar.
$

The zero locus
$F^{-1}(0) = \Spec \Sbar$
has an action of $K = \Hom(L, \bG_m)$
coming from the $L$-grading on $\Sbar$,
and the quotient stack
$
 Y = [(F^{-1}(0) \setminus \bszero) / K]
$
of the complement of the origin $\bszero \in F^{-1}(0)$
by this action is a Calabi-Yau orbifold.
The {\em Berglund-H\"{u}bsch transpose} of $Y$
is another Calabi-Yau hypersurface $\Yv$ in a weighted projective space
defined by the Berglund-H\"{u}bsch transpose $\Fv$ of $F$.
It is conjectured \cite{Berglund-Hubsch}
that $(Y, \Yv)$ is a mirror pair,
so that there is an equivalence
\begin{align} \label{eq:hms_CY}
 D^b \coh Y \cong D^\pi \Fuk \Yv
\end{align}
of triangulated categories
between the derived category of coherent sheaves on $Y$
and the split-closed derived Fukaya category of $\Yv$
\cite{Kontsevich_HAMS, Seidel_K3}.

The weighted projective hypersurface $\Yv$ is
a compactification of the Milnor fiber
$
 \Yvo = \Yv \setminus \{ x_{d+2} = 0 \}
  \cong \fv^{-1}(-1)
$
of $\fv$.
Let $(V_i)_{i=1}^m$ be a distinguished basis
of vanishing cycles of $\fv$,
considered as Lagrangian submanifolds of $\Yvo$.
Let further $\dirscF$ be the directed subcategory of $\Fuk \Yvo$
consisting of $(V_i)_{i=1}^m$;
it is an $A_\infty$-subcategory with
$$
 \hom_{\dirscF}(V_i, V_j) =
\begin{cases}
 \bsk \cdot \id_{V_i} & i = j, \\
 \hom_{\Fuk \Yvo}(V_i, V_j) & i < j, \\
 0 & \text{otherwise},
\end{cases}
$$
and $A_\infty$-operations on $\dirscF$ are
inherited from those of $\Fuk \Yvo$ in the obvious way.
The Picard-Lefschetz theory
\cite[Theorem 18.24]{Seidel_PL} gives a derived equivalence
$$
 D^b \dirscF \cong D^b \Fuk \fv
$$
of $\dirscF$
and the Fukaya-Seidel category of $\fv$.
Although the Fukaya category $\Fuk \Yv$ is a deformation
of $\Fuk \Yvo$
\cite{Seidel_FCD}
and the $A_\infty$-operations on $\Fuk \Yv$
is difficult to compute explicitly,
Poincar\'{e} duality tells us
that the space of morphisms in $\Fuk \Yv$
is given by the {\em Calabi-Yau completion}
\begin{align} \label{eq:duality1}
 \Hom_{\Fuk \Yv}^i(V_j, V_k)
  = \Hom_{\dirscF}^i(V_j, V_k) \oplus \Hom_{\dirscF}^{d-i}(V_k, V_j)^\vee
\end{align}
of the space of morphisms in $\dirscF$.

If we assume homological mirror symmetry
\eqref{eq:hms_sing}
for singularities,
then $D^b_\sing(\gr \Rbar)$ has
a full exceptional collection $(E_i)_{i=1}^m$
corresponding to $V_i$,
and Theorem \ref{th:main} shows that
\begin{itemize}
 \item
the full subcategory of
$D^b_\sing(\gr \Sbar) \cong D^b \coh Y$
consisting of $(\Phi_\sing(E_i))_{i=1}^m$
is a Calabi-Yau completion
of the full subcategory $\dirscE$ of $D^b_\sing(\gr \Rbar)$
consisting of $(E_i)_{i=1}^m$, and
 \item
$\dirscE$ is the directed subcategory of $D^b_\sing(\gr \Sbar)$
consisting of $(\Phi_\sing(E_i))_{i=1}^m$.
\end{itemize}
Now the compatibility of homological mirror symmetry
for Calabi-Yau manifolds and that for singularities
is the existence of a commutative diagram
\begin{align} \label{eq:diagram}
\begin{array}{ccc}
 \dirscF & \hookrightarrow & D^\pi \Fuk \Yv \\
  \rotatebox{90}{$\cong$} & & \rotatebox{90}{$\cong$} \\
 \dirscE & \hookrightarrow & D^b  \coh Y
\end{array}
\end{align}
where horizontal arrows are embeddings of directed subcategories
and vertical equivalences are homological mirror symmetry.
Moreover,
we expect that the images of the horizontal arrows
split-generate the categories on the right,
so that one can divide
the proof of homological mirror symmetry
for Calabi-Yau manifolds
into two steps;
the first step is the proof of homological mirror symmetry
for singularities,
and the second step is the analysis of Calabi-Yau completion.
This is analogous
to the proof of homological mirror symmetry
for Calabi-Yau hypersurfaces
in projective spaces,
which can be interpreted
as first proving homological mirror symmetry
for the ambient projective space,
and then passing to the Calabi-Yau hypersurface
by taking a non-trivial Calabi-Yau completion
\cite{Seidel_K3,Sheridan_CY,Nohara-Ueda_HMSQ}.

This paper is organized as follows:
In Section \ref{sc:gmf},
we recall basic definitions on matrix factorizations.
In Section \ref{sc:pmf},
we give an explicit description of the push-forward
in terms of matrix factorizations,
which will be used in Section \ref{sc:mpf}
to prove Theorem \ref{th:main}.

{\bf Acknowledgment:}
We thank Yukinobu Toda for pointing out an error
in \eqref{eq:ext_coh} in the earlier version.
This research is supported by Grant-in-Aid for Young Scientists
(No.~24740043).

\section{Graded matrix factorizations} \label{sc:gmf}

Let $R$ be a graded regular ring with $a$-invariant $r$
and Krull dimension $d$, and
$f \in R_h$ be a homogeneous element of degree $h$.
We assume that $R$ is graded by $\bZ$
for simplicity of exposition,
although our discussion easily generalizes
to grading by any abelian group.

A {\em graded matrix factorization}
\begin{align*}
 \\[-3mm]
\begin{psmatrix}[mnode=r]
 [ref=r] \bigoplus_{i=1}^m R(e_i)
  & &
 [ref=l] \bigoplus_{i=1}^m R(d_i)
\end{psmatrix}
\psset{shortput=nab,arrows=->,labelsep=3pt,nodesep=3pt,offset=4pt}
\ncarc[arcangle=20]{1,1}{1,3}^{\phi}
\ncarc[arcangle=20]{1,3}{1,1}^{\psi}
 \\[0mm]
\end{align*}
of $f$ over $R$ consists of
\begin{itemize}
 \item
a pair
$
 M_0 = \bigoplus_{i=1}^m R(d_i)
$
and
$
 M_1 = \bigoplus_{i=1}^m R(e_i)
$
of graded free $R$-module
and
 \item
a pair $(\phi, \psi)$ of morphisms
$
 \phi : M_1 \to M_0
$
and
$
 \psi : M_0(-h) \to M_1
$
of graded modules
\end{itemize}
satisfying
$
 \phi \circ \psi = f \cdot \id_{M^0}
$
and
$
 \psi \circ \phi = f \cdot \id_{M^1}.
$
The integers $d_i$ and $e_i$ are shifts
in the internal grading of $R$-modules,
and the natural number $m = \rank M^0 = \rank M^1$ is
called the {\em rank} of the matrix factorization.

A {\em morphism} of graded matrix factorizations
from $(\phi, \psi) : M_1 \rightleftarrows M_0$
to $(\phi', \psi') : M_1' \rightleftarrows M_0'$
is a pair $(\alpha, \beta)$ of morphisms
$
 \alpha : M_0 \to M'_0
$
and
$
 \beta : M_1 \to M'_1
$
making the diagram
$$
\begin{CD}
 M_0(-h) @>{\psi}>> M_1@>{\phi}>> M_0 \\
  @V{\alpha}VV @V{\beta}VV @V{\alpha}VV \\
 M'_0(-h) @>{\psi'}>> M'_1@>{\phi'}>> M'_0
\end{CD}
$$
commute;
$
 \alpha \circ \phi = \phi' \circ \beta
$
and
$
 \beta \circ \psi = \psi' \circ \alpha.
$
Here, the morphism $M_0(-h) \to M'_0(-h)$
of graded $R$-modules
corresponding to the morphism $\alpha : M_0 \to M'_0$
is denoted by the same symbol.

Two morphisms $(\alpha, \beta)$ and $(\alpha', \beta')$
are {\em homotopic}
if there exist morphisms
$
 \xi : M_0 \to M'_1
$
and
$
 \eta : M_1 \to M'_0
$
satisfying
$
 \alpha - \alpha' = \phi' \circ \xi + \eta \circ \psi
$
and
$
 \beta - \beta' = \psi' \circ \eta + \xi \circ \phi.
$

The {\em homotopy category of graded matrix factorizations}
of $f$ over $R$
is the category $\Hmf_R^{\gr} f$
whose objects are graded matrix factorizations of $f$ over $R$ and
whose morphisms are morphisms of graded matrix factorizations
up to homotopy.

The category $\Hmf_R^{\gr} f$ is equivalent
to the {\em stable category}
$\underline{\CM}^{\gr} \Rbar$
of graded Cohen-Macaulay modules
over the quotient ring $\Rbar = R / (f)$
\cite{Eisenbud_HACI},
which in turn is equivalent to the graded stable derived category
$D^b_\sing(\gr \Rbar)$
of $\Rbar$
\cite{
Buchweitz_MCM,
Orlov_TCS}.

\section{The push-forward of a matrix factorization}
 \label{sc:pmf}

We keep the notations from Section \ref{sc:gmf}.
Let $S = R \otimes \bsk[w]$ be another graded ring
with $\deg w = a$,
and $F = f + w g$ be a homogeneous element of $S$
such that $g \in w S$.
The corresponding quotient ring will be denoted
by $\Sbar = S / (F)$,
and one has natural injections $R \hookrightarrow S$ and
$\Rbar \hookrightarrow \Sbar$.
The natural ring homomorphism
$\Sbar \to \Sbar / (w) \cong \Rbar$
induces the push-forward functor
$$
 \Phi_\sing : D^b_\sing(\gr \Rbar)
  \to D^b_\sing(\gr \Sbar),
$$
since $\Rbar$ is perfect as an $\Sbar$-module.
\begin{proposition} \label{pr:mf_push-forward}
The functor
$$
 \Phi_{\Hmf} : \Hmf_R^{\gr} f \to \Hmf_S^{\gr} F
$$
corresponding to the push-forward functor
$$
 \Phi_\sing : D^b_\sing(\gr \Rbar)
  \to D^b_\sing(\gr \Sbar)
$$
sends the graded matrix factorization
\begin{align} \label{eq:mf1}
\begin{CD}
 \bigoplus_{i=1}^m R(d_i - h)
  @>{\psi}>>
 \bigoplus_{i=1}^m R(e_i)
  @>{\phi}>>
 \bigoplus_{i=1}^m R(d_i) 
\end{CD}
\end{align}
of $f$ over $R$
to the graded matrix factorization
\begin{align} \label{eq:mf2}
\begin{CD}
\begin{array}{c}
 \bigoplus_{i=1}^m S(d_i-h) \\ \oplus \\ \bigoplus_{i=1}^m S(e_i-a)
\end{array}
 @>{
\psitilde =
\scriptsize
\begin{pmatrix}
 \psi & -w \\
 g & \phi
\end{pmatrix}
 }>>
\begin{array}{c}
 \bigoplus_{i=1}^m S(e_i) \\ \oplus \\ \bigoplus_{i=1}^m S(d_i-a)
\end{array}
 @>{
\phitilde =
\scriptsize
\begin{pmatrix}
 \phi &  w \\
 -g & \psi
\end{pmatrix}
}>>
\begin{array}{c}
 \bigoplus_{i=1}^m S(d_i) \\ \oplus \\ \bigoplus_{i=1}^m S(e_i+h-a),
\end{array}
\end{CD}
\end{align}
of $F$ over $S$,
where $\phi, \psi \in R$ are considered as elements of $S$
by the injection $R \hookrightarrow S$.
\end{proposition}

\begin{proof}
Recall from \cite{Eisenbud_HACI}
that the matrix factorization \eqref{eq:mf1} corresponds
to the $\Rbar$-module
$\Mbar = \coker \phibar$
through the 2-periodic projective resolution
\begin{align} \label{eq:resolution1}
\begin{CD}
 \cdots
  @>{\psibar}>>
 \bigoplus_{i=1}^m \Rbar(e_i - h)
  @>{\phibar}>>
 \bigoplus_{i=1}^m \Rbar(d_i - h)
  @>{\psibar}>>
 \bigoplus_{i=1}^m \Rbar(e_i) \\
  @.
  @>{\phibar}>>
 \bigoplus_{i=1}^m \Rbar(d_i) 
  @>>>
 \Mbar
  @>>>
 0,
\end{CD}
\end{align}
where $\psibar$ and $\phibar$ are morphisms of $\Rbar$-modules
induced by $\psi$ and $\phi$.
By replacing each free $\Rbar$-modules
in \eqref{eq:resolution1}
with their $\Sbar$-free resolutions
$$
 0 \to \Sbar(\ell-a) \xto{w} \Sbar(\ell) \to \Rbar(\ell) \to 0,
$$
one obtains a projective resolution
\begin{align*}
\begin{CD}
 \cdots
 @>{\psibartilde}>>
\begin{array}{c}
 \bigoplus_{i=1}^m \Sbar(e_i-h) \\ \oplus \\ \bigoplus_{i=1}^m \Sbar(d_i-a-h),
\end{array}
 @>{\phibartilde}>>
\begin{array}{c}
 \bigoplus_{i=1}^m \Sbar(d_i-h) \\ \oplus \\ \bigoplus_{i=1}^m \Sbar(e_i-a)
\end{array} \\
 @>{\psibartilde}>>
\begin{array}{c}
 \bigoplus_{i=1}^m \Sbar(e_i) \\ \oplus \\ \bigoplus_{i=1}^m \Sbar(d_i-a)
\end{array}
 @>{\scriptsize
\begin{pmatrix}
 \phibar &  w
\end{pmatrix}
}>>
\begin{array}{c}
 \bigoplus_{i=1}^m \Sbar(d_i)
\end{array}
 @>>> \Mbar @>>> 0
\end{CD}
\end{align*}
of $\Mbar$
as an $\Sbar$-module,
where $\phibartilde$ and $\psibartilde$
are morphisms of $\Sbar$-modules
induced by $\phitilde$ and $\psitilde$.
This complex is 2-periodic except for the first two terms,
and clearly corresponds to the matrix factorization
\eqref{eq:mf2},
so that Proposition \ref{pr:mf_push-forward} is proved.
\end{proof}

\section{Morphisms between push-forwards}
 \label{sc:mpf}

We keep the notations from Section \ref{sc:pmf}.
The following proposition gives Theorem \ref{th:main}
written in terms of matrix factorizations.

\begin{proposition} \label{pr:hom}
Let $\scE = (\phi, \psi)$ and $\scE' = (\phi', \psi')$
be matrix factorizations of $f$ over $R$ and
$\scF = \Phi_{\Hmf}(\scE)$ and $\scF' = \Phi_{\Hmf}(\scE')$ be their push-forwards.
Then one has
$$
 \Hom(\scF, \scF')
  \cong \Hom(\scE, \scE')
   \oplus \Hom(\scE', \scE(r+h-a))^\vee[-d].
$$
\end{proposition}

\begin{proof}
An element of $\Hom(\scF, \scF')$ is represented
by a pair $(\alpha, \beta)$ of morphisms
\begin{align*}
 \alpha = 
\begin{pmatrix}
 \alpha_1 & \alpha_2 \\
 \alpha_3 & \alpha_4
\end{pmatrix}
 :
\begin{array}{c}
 \bigoplus_{i=1}^m S(d_i) \\ \oplus \\ \bigoplus_{i=1}^m S(e_i+h-a)
\end{array}
 \lto
\begin{array}{c}
 \bigoplus_{i=1}^m S(d'_i) \\ \oplus \\ \bigoplus_{i=1}^m S(e'_i+h-a)
\end{array}
\end{align*}
and
\begin{align*}
 \beta = 
\begin{pmatrix}
 \beta_1 & \beta_2 \\
 \beta_3 & \beta_4
\end{pmatrix}
 :
\begin{array}{c}
 \bigoplus_{i=1}^m S(e_i) \\ \oplus \\ \bigoplus_{i=1}^m S(d_i-a)
\end{array}
 \lto
\begin{array}{c}
 \bigoplus_{i=1}^m S(e'_i) \\ \oplus \\ \bigoplus_{i=1}^m S(d'_i-a)
\end{array}
\end{align*}
making the diagram
\begin{align*}
\begin{CD}
\begin{array}{c}
 \bigoplus_{i=1}^m S(d_i-h) \\ \oplus \\ \bigoplus_{i=1}^m S(e_i-a)
\end{array}
 @>{\psitilde}>>
\begin{array}{c}
 \bigoplus_{i=1}^m S(e_i) \\ \oplus \\ \bigoplus_{i=1}^m S(d_i-a)
\end{array}
 @>{\phitilde}>>
\begin{array}{c}
 \bigoplus_{i=1}^m S(d_i) \\ \oplus \\ \bigoplus_{i=1}^m S(e_i+h-a)
\end{array} \\
 @V{\alpha}VV
 @V{\beta}VV
 @V{\alpha}VV \\
\begin{array}{c}
 \bigoplus_{i=1}^m S(d'_i-h) \\ \oplus \\ \bigoplus_{i=1}^m S(e'_i-a)
\end{array}
 @>{\psitilde'}>>
\begin{array}{c}
 \bigoplus_{i=1}^m S(e'_i) \\ \oplus \\ \bigoplus_{i=1}^m S(d'_i-a)
\end{array}
 @>{\phitilde'}>>
\begin{array}{c}
 \bigoplus_{i=1}^m S(d'_i) \\ \oplus \\ \bigoplus_{i=1}^m S(e'_i+h-a)
\end{array} 
\end{CD}
\end{align*}
%
%
commute;
$
  \alpha \circ \phitilde = \phitilde' \circ \beta
$
and
$
  \beta \circ \psitilde = \psitilde' \circ \alpha.
$
The former condition can be written explicitly as the equality
\begin{align} \label{eq:cocycle1}
\begin{pmatrix}
 \alpha_1 \phi - \alpha_2 g & \alpha_1 w + \alpha_2 \psi \\
 \alpha_3 \phi - \alpha_4 g & \alpha_3 w + \alpha_4 \psi
\end{pmatrix}
 =
\begin{pmatrix}
 \phi' \beta_1 + w \beta_3 & \phi' \beta_2 + w \beta_4 \\
 - g \beta_1 + \psi' \beta_3 & - g \beta_2 + \psi' \beta_4
\end{pmatrix}
\end{align}
where
\begin{align*}
 \alpha \circ \phitilde =
\begin{pmatrix}
 \alpha_1 & \alpha_2 \\
 \alpha_3 & \alpha_4 
\end{pmatrix}
\begin{pmatrix}
 \phi & w \\
 - g & \psi 
\end{pmatrix}
 &=
\begin{pmatrix}
 \alpha_1 \phi - \alpha_2 g & \alpha_1 w + \alpha_2 \psi \\
 \alpha_3 \phi - \alpha_4 g & \alpha_3 w + \alpha_4 \psi
\end{pmatrix}
\end{align*}
and
\begin{align*}
 \phitilde' \circ \beta =
\begin{pmatrix}
 \phi' & w \\
 -g & \psi'
\end{pmatrix}
\begin{pmatrix}
 \beta_1 & \beta_2 \\
 \beta_3 & \beta_4
\end{pmatrix}
 &=
\begin{pmatrix}
 \phi' \beta_1 + w \beta_3 & \phi' \beta_2 + w \beta_4 \\
 - g \beta_1 + \psi' \beta_3 & - g \beta_2 + \psi' \beta_4
\end{pmatrix},
\end{align*}
and similarly as the equality
\begin{align} \label{eq:cocycle2}
\begin{pmatrix}
 \beta_1 \psi + \beta_2 g & - \beta_1 w + \beta_2 \phi \\
 \beta_3 \psi + \beta_4 g & - \beta_3 w + \beta_4 \phi
\end{pmatrix}
 =
\begin{pmatrix}
 \psi' \alpha_1 - w \alpha_3 & \psi' \alpha_2 - w \alpha_4 \\
 g \alpha_1 + \phi' \alpha_3 & g \alpha_2 + \phi' \alpha_4
\end{pmatrix}
\end{align}
where
\begin{align*}
 \beta \circ \psitilde =
\begin{pmatrix}
 \beta_1 & \beta_2 \\
 \beta_3 & \beta_4 
\end{pmatrix}
\begin{pmatrix}
 \psi & -w \\
 g & \phi 
\end{pmatrix}
 &=
\begin{pmatrix}
 \beta_1 \psi + \beta_2 g & - \beta_1 w + \beta_2 \phi \\
 \beta_3 \psi + \beta_4 g & - \beta_3 w + \beta_4 \phi
\end{pmatrix}
\end{align*}
and
\begin{align*}
 \psitilde' \circ \alpha =
\begin{pmatrix}
 \psi' & - w \\
 g & \phi'
\end{pmatrix}
\begin{pmatrix}
 \alpha_1 & \alpha_2 \\
 \alpha_3 & \alpha_4
\end{pmatrix}
 &=
\begin{pmatrix}
 \psi' \alpha_1 - w \alpha_3 & \psi' \alpha_2 - w \alpha_4 \\
 g \alpha_1 + \phi' \alpha_3 & g \alpha_2 + \phi' \alpha_4
\end{pmatrix}
\end{align*}
for the latter.
Two morphisms $(\alpha, \beta)$ and $(\alpha', \beta')$
are homotopic if there exist morphisms
\begin{align*}
 \xi = 
\begin{pmatrix}
 \xi_1 & \xi_2 \\
 \xi_3 & \xi_4
\end{pmatrix}
 :
\begin{array}{c}
 \bigoplus_{i=1}^m S(d_i) \\ \oplus \\ \bigoplus_{i=1}^m S(e_i+h-a)
\end{array}
 \to
\begin{array}{c}
 \bigoplus_{i=1}^m S(e'_i) \\ \oplus \\ \bigoplus_{i=1}^m S(d'_i-a)
\end{array}
\end{align*}
and
\begin{align*}
 \eta = 
\begin{pmatrix}
 \eta_1 & \eta_2 \\
 \eta_3 & \eta_4
\end{pmatrix}
 :
\begin{array}{c}
 \bigoplus_{i=1}^m S(e_i) \\ \oplus \\ \bigoplus_{i=1}^m S(d_i-a)
\end{array}
 \to
\begin{array}{c}
 \bigoplus_{i=1}^m S(d'_i) \\ \oplus \\ \bigoplus_{i=1}^m S(e'_i+h-a)
\end{array}
\end{align*}
satisfying
$
 \alpha - \alpha' = \phitilde' \circ \xi + \eta \circ \psitilde
$
and
$
 \beta - \beta' = \psitilde' \circ \eta + \xi \circ \phitilde.
$
%
The terms on the right hand sides
can be explicitly written as
\begin{align*}
 \phitilde' \circ \xi =
\begin{pmatrix}
 \phi' & w \\
 -g & \psi'
\end{pmatrix}
\begin{pmatrix}
 \xi_1 & \xi_2 \\
 \xi_3 & \xi_4
\end{pmatrix}
 &=
\begin{pmatrix}
 \phi' \xi_1 + w \xi_3 & \phi' \xi_2 + w \xi_4 \\
 - g \xi_3 + \psi' \xi_3 & - g \xi_2 + \psi' \xi_4
\end{pmatrix}, \\
 \eta \circ \psitilde =
\begin{pmatrix}
 \eta_1 & \eta_2 \\
 \eta_3 & \eta_4
\end{pmatrix}
\begin{pmatrix}
 \psi & -w \\
 g & \phi
\end{pmatrix}
 &=
\begin{pmatrix}
 \eta_1 \psi + \eta_2 g & - \eta_1 w + \eta_2 \phi \\
 \eta_3 \psi - \eta_4 g & - \eta_3 w + \eta_4 \phi
\end{pmatrix}, \\
 \psitilde' \circ \eta =
\begin{pmatrix}
 \psi' & -w \\
 g & \phi'
\end{pmatrix}
 \circ
\begin{pmatrix}
 \eta_1 & \eta_2 \\
 \eta_3 & \eta_4
\end{pmatrix}
 &=
\begin{pmatrix}
  \psi' \eta_1 - w \eta_3 & \psi' \eta_2 - w \eta_4 \\
  g \eta_1 + \phi' \eta_3 & g \eta_2 + \phi' \eta_4
\end{pmatrix}, \\
 \xi \circ \phitilde =
\begin{pmatrix}
 \xi_1 & \xi_2 \\
 \xi_3 & \xi_4
\end{pmatrix}
\begin{pmatrix}
 \phi & w \\
 -g & \psi
\end{pmatrix}
 &=
\begin{pmatrix}
 \xi_1 \phi - \xi_2 g & \xi_1 w + \xi_2 \psi \\
 \xi_3 \phi - \xi_4 g & \xi_3 w + \xi_4 \psi
\end{pmatrix}.
\end{align*}
Note that one can remove $w$-dependence
of $\alpha_1$
by choosing $\xi_3$ in a suitable way
and achieve $\alpha_1 \in \Mat_k(R)$.
Similarly,
$w$-dependence of $\alpha_2$,
$\beta_1$ and $\beta_2$ can be removed
by choosing $\xi_4$,
$\eta_3$ and $\eta_4$ respectively
in a suitable way.
Moreover,
these operation
of removing $w$-dependence
can be performed independently,
so that one can achieve
$\alpha_1, \alpha_2, \beta_1, \beta_2 \in \Mat_k(R)$
simultaneously.

Now the $(1,1)$-component of \eqref{eq:cocycle1}
gives the equation
$$
 \alpha_1 \phi - \alpha_2 g = \phi' \beta_1 + w \beta_3.
$$
Since $\alpha_1, \beta_1, \phi, \phi' \in \Mat_k(R)$
and $\alpha_2 g, w \beta_3 \in w \Mat_k(S)$
in the direct sum decomposition
$\Mat_k(S) = \Mat_k(R) \oplus w \Mat_k(S)$,
one obtains
\begin{align} \label{eq:cocycle2-1}
 \alpha_1 \phi = \phi' \beta_1
  \qquad \text{and} \qquad
 \beta_3 = - \frac{g}{w} \alpha_2.
\end{align}
Similarly,
from the $(1,2)$-component of \eqref{eq:cocycle1},
one obtains
$$
 \alpha_1 w + \alpha_2 \psi = \phi' \beta_2 + w \beta_4
$$
which,
together with $\alpha_2, \beta_2, \psi, \phi' \in \Mat_k(R)$,
implies
\begin{align} \label{eq:cocycle2-2}
 \alpha_2 \psi = \phi' \beta_2
  \qquad \text{and} \qquad
 \beta_4 = \alpha_1.
\end{align}
The same argument for $(1,1)$- and $(1,2)$-components
of \eqref{eq:cocycle2} gives
\begin{align} \label{eq:cocycle2-3}
 \beta_1 \psi = \psi' \alpha_1
  \qquad \text{and} \qquad
 \alpha_3 = - \frac{g}{w} \beta_2
\end{align}
and
\begin{align} \label{eq:cocycle2-4}
 \beta_2 \phi = \psi' \alpha_2
  \qquad \text{and} \qquad
 \alpha_4 = \beta_1
\end{align}
respectively.
This determines $\alpha$ and $\beta$ completely
from $(\alpha_1, \alpha_2, \beta_1, \beta_2)$
satisfying
\begin{align}
 \alpha_1 \phi &= \phi' \beta_1, &
 \alpha_2 \psi &= \phi' \beta_2, &
 \beta_1 \psi &= \psi' \alpha_1, &
 \beta_2 \phi &= \psi' \alpha_2.
\end{align}
Other components of \eqref{eq:cocycle1}
and \eqref{eq:cocycle2}
are automatically satisfied
for $\alpha$ and $\beta$
satisfying \eqref{eq:cocycle2-1},
\eqref{eq:cocycle2-2},
\eqref{eq:cocycle2-3}, and
\eqref{eq:cocycle2-4}.
This shows that any element of $\Hom(\scF, \scF')$
can be represented by a pair $(\alpha, \beta)$ of morphisms
coming from a representative
$(\alpha_1, \beta_1)$ of $\Hom(\scE, \scE')$
and a representative
$(\alpha_2, \beta_2)$ of $\Hom(\scE[1](-a), \scE')$.

If two morphisms from $\scF$ to $\scF'$
represented by pairs $(\alpha, \beta)$ and
$(\alpha', \beta')$
coming from representatives
$(\alpha_1, \beta_1)$ and $(\alpha_1', \beta_1')$
of $\Hom(\scE, \scE')$
and representatives
$(\alpha_2, \beta_2)$ and $(\alpha_2', \beta_2')$
of $\Hom(\scE[1](-a), \scE')$
are homotopic,
then one has two morphisms $\xi$ and $\eta$
satisfying
$
 \alpha - \alpha' = \phitilde' \circ \xi + \eta \circ \psitilde
$
and
$
 \beta - \beta' = \psitilde' \circ \eta + \xi \circ \phitilde.
$
By looking at the $\Mat_{2k}(R)$ component
of these equations
with respect to the direct sum decomposition
$\Mat_{2k}(S) = \Mat_{2k}(R) \oplus w \Mat_{2k}(S)$,
one finds that
$(\alpha_1, \beta_1)$ and $(\alpha_1', \beta_1')$ are homotopic over $R$,
and
$(\alpha_2, \beta_2)$ and $(\alpha_2', \beta_2')$ are homotopic over $R$
as well.
It follows that one has an isomorphism
$$
 \Hom(\scF, \scF')
  \cong \Hom(\scE, \scE') \oplus \Hom(\scE[1](-a), \scE')
$$
of spaces of morphisms.
Graded Auslander-Reiten duality
\cite{Auslander-Reiten_ASSZR}
implies Serre duality 
\begin{align*}
 \Hom(\scM, \scN)
  \cong \Hom(\scN, \scM(r+h)[d-1])^\vee
\end{align*}
in the graded stable derived category of $\Rbar$
\cite[Corollary 2.5]{Iyama-Takahashi_TCTQS},
so that
\begin{align*}
\Hom(\scE[1](-a), \scE')
  &\cong \Hom(\scE', \scE(r+h-a))^\vee[-d]
\end{align*}
and Proposition \ref{pr:hom} is proved.
\end{proof}

\bibliographystyle{amsalpha}
\bibliography{bibs}

\noindent
Kazushi Ueda

Department of Mathematics,
Graduate School of Science,
Osaka University,
Machikaneyama 1-1,
Toyonaka,
Osaka,
560-0043,
Japan.

{\em e-mail address}\ : \  kazushi@math.sci.osaka-u.ac.jp
\ \vspace{0mm} \\

\end{document}